\documentclass[12pt,a4paper]{article}
\usepackage[english]{babel}

\usepackage{pstricks}
\usepackage{pst-node}
\usepackage{amsmath,boxedminipage}
\usepackage{amssymb}
\usepackage{graphicx}
\usepackage{enumerate}
\usepackage{color}
\usepackage[text={15cm,24cm}]{geometry}
\usepackage[normalem]{ulem}
\usepackage[ansinew]{inputenc}
\usepackage{ulem}
\usepackage{float}
\usepackage{subcaption}
\usepackage{diagbox}
\usepackage{booktabs}

\usepackage{thmtools}
\usepackage{thm-restate}
\usepackage{hyperref}
\usepackage{cleveref}
\newcounter{theoremcpt}
\declaretheorem[name=Theorem,numberlike=theoremcpt]{theorem}

\newenvironment{proof}{{\bf Proof}:\ }%
   {~\ \hfill $\Box$\vspace{0,5cm}}

    {~\ \hfill$\Box$\vspace{0,5cm}}
\newenvironment{ack}{\vskip5mm{\bf Acknowledgements:}}%

\newtheorem{conj}{Conjecture}[theoremcpt]
\newtheorem{claim}{Claim}
\newtheorem{coro}[theoremcpt]{Corollary}

\newcounter{claimish}[theorem]

\graphicspath{{.}{graphics/}}
\newrgbcolor{lightlightlightgray}{0.9 0.9 0.9}

\numberwithin{equation}{section}

\begin{document}
\title{The bondage number of chordal graphs}

\author{V. Bouquet\footnotemark[1]}
\date{\today}

\def\thefootnote{\fnsymbol{footnote}}

\footnotetext[1]{ \noindent
Conservatoire National des Arts et M\'etiers, CEDRIC laboratory, Paris (France). Email: {\tt
valentin.bouquet@cnam.fr}}

\maketitle

\begin{abstract}
   A set $S\subseteq V(G)$ of a graph $G$ is a dominating set if each vertex has a neighbor in $S$ or belongs to $S$. Let $\gamma(G)$ be the cardinality of a minimum dominating set in $G$. The bondage number $b(G)$ of a graph $G$ is the smallest cardinality of a set edges $A\subseteq E(G)$ such that $\gamma(G-A)=\gamma(G)+1$. A chordal graph is a graph with no induced cycle of length four or more.

   In this paper, we prove that the bondage number of a chordal graph $G$ is at most the order of its maximum clique, that is, $b(G)\leq \omega(G)$. We show that this bound is best possible.

\vspace{0.2cm}
\noindent{\textbf{Keywords}\/}: Bondage number, domination, chordal graphs, maximum clique.
\end{abstract}

\section{Introduction}\label{intro}

Given a graph $G=(V,E)$, a set $S\subseteq V$ is called a \textit{dominating set} if every vertex $v\in V$ is an element of $S$ or is adjacent to an element of $S$. The minimum cardinality of a dominating set in $G$ is called the \textit{domination number} and is denoted by $\gamma(G)$. A dominating set $S\subseteq V$, with $\vert S\vert=\gamma(G)$, is called a \textit{minimum dominating set}. For an overview of the topics in graph domination, we refer to the book of Haynes et al.\ \cite{Haynes}. The \textit{bondage number} has been introduced by Fink et al.\ in \cite{Fink} has a parameter to measure the criticality of a graph with respect to the domination number. The bondage number $b(G)$ of a graph $G$ is the minimum number of edges whose removal from $G$ increases the domination number, that is, with $E'\subseteq E(G)$ such that $\gamma(G-E') = \gamma(G)+1$. To this date, the bondage number and related properties have been extensively studied. We refer to the survey of Xu \cite{Xu} for an extending overview of the bondage number and its related properties. One result we would like to highlight is a tight upper bound on the bondage number of trees. It has been discovered independently by Bauer et al.\ in \cite{Bauer} and Fink et al.\ in \cite{Fink}.

\begin{theorem}[\cite{Bauer}\cite{Fink}]\label{BondTree}
   If $G$ is a tree, then $b(G)\leq 2$.
\end{theorem}

We would also like to point out an upper bound on the bondage number of block graphs. The \textit{block graphs} are the chordal diamond-free graphs (a diamond is a clique of order four minus an edge). The following upper bound on block graphs has been shown by Teschner in \cite{Teschner}.

\begin{theorem}[\cite{Teschner}]\label{BlockGraph}
   If $G$ is a block graph, then $b(G)\leq \Delta(G)$.
\end{theorem}

In this paper, we prove the following upper bound that encapsulates Theorem \ref{BondTree}, and is a stronger statement than the one of Theorem \ref{BlockGraph}.

\begin{restatable}{theorem}{BondChordal}\label{BondChordal}
   Let $G$ be a chordal graph. If $G$ is a clique, then $b(G) = \lceil \omega(G)/2 \rceil$. Else $b(G)\leq \omega(G)\leq \Delta(G)$.
\end{restatable}

\section{Preliminaries}

The graphs considered in this paper are finite and simple, that is, without directed edges or loops or parallel edges. The reader is referred to \cite{Bondy} for definitions and notations in graph theory. \medskip

Let $G=(V,E)$ be a graph with vertex set $V=V(G)$ and edge set $E=E(G)$. Let $v\in V(G)$ and $xy\in E(G)$. We say that $x$ and $y$ are the \textit{endpoints} of the edge. Let $\delta(G)$ and $\Delta(G)$ denote its \textit{minimum degree} and its \textit{maximum degree}, respectively. The \textit{degree} of $v$ in $G$ is $d_G(v)$ or simply $d(v)$ when the referred graph is obvious. If $d(v)=0$, we say that $v$ is \textit{isolated} in $G$. We denote by $d(u,v)$ the \textit{distance} between two vertices, that is, the length of a shortest path between $u$ and $v$. Note that when $uv\in E$, $d(u,v)=1$. We denote by $N_G(v)$ the \textit{open neighborhood} of a vertex $v$ in $G$, and $N_G[v]=N_G(v)\cup\{v\}$ its \textit{closed neighborhood} in $G$. When it is clear from context, we write $N(v)$ and $N[v]$. The \textit{open neighborhood} of a set $U\subseteq V$ is $N(U)=\{N(u)\setminus U \mid u\in U\}$. For a subset $U\subseteq V$, let $G[U]$ denote the \textit{subgraph} of $G$ induced by $U$ which has vertex set $U$ and edge set $\{uv\in E \mid u,v\in U\}$. We may refer to $U$ as an \textit{induced subgraph} of $G$ when it is clear from the context. If a graph $G$ has no induced subgraph isomorphic to a fixed graph $H$, we say that $G$ is $H$-free.
For $n\geq 1$, the graph $P_n=u_1-u_2-\cdots-u_n$ denotes the \textit{cordless path} or \textit{induced path} on $n$ vertices, that is, $V({P_n})=\{u_1,\ldots,u_n\}$ and $E({P_n})=\{u_iu_{i+1}\; |\; 1\leq i\leq n-1\}$. For $n\geq 3$, the graph $C_n$ denotes the \textit{cordless cycle} or \textit{induced cycle} on $n$ vertices, that is, $V({C_n})=\{u_1,\ldots,u_n\}$ and $E({C_n})=\{u_iu_{i+1}\; |\; 1\leq i\leq n-1\}\cup \{u_nu_1\}$. For $n\ge 4$, $C_n$ is called a \textit{hole}.
A set $U\subseteq V$ is called a \textit{clique} if any pairwise distinct vertices $u,v\in U$ are adjacent. We denote by $\omega(G)$ the size of a maximum clique in $G$. The graph $K_n$ is the clique with $n$ vertices. A set $U\subseteq V$ is called a \textit{stable set} or an \textit{independent set} if any pairwise distinct vertices $u,v\in U$ are non adjacent. \medskip

We recall the two following results on the upper bound of the bondage number. They will be of use to prove Theorem \ref{BondChordal} in the next section.

\begin{theorem}[Fink et al.\ \cite{Fink}]\label{ubound}
   Let $G=(V,E)$ be a graph, and $u,v\in V$ such that $d(u,v)\leq 2$. Then $b(G) \leq d(u) + d(v) - 1$.
\end{theorem}

\begin{theorem}[Hartnell and Rall \cite{Hartnell}]\label{uboundneigh}
   Let $G=(V,E)$ be a graph, and $uv\in E$. Then $b(G) \leq d(u) + d(v) - 1 - \vert N(u) \cap N(v)\vert$.
\end{theorem}

\section{Chordal graphs}

A \textit{chordal graph} is a graph that has no hole. Stated otherwise, every subgraph that is a cycle of length at least four has a chord. We prove our main Theorem.

\BondChordal*

\begin{proof}
   We can assume that $G$ is connected with at least two vertices. Note that $\Delta(G)\geq \omega(G)-1$ and $\Delta(G)=\omega(G)-1$ if and only if $G$ is a clique. When $G$ is an even clique, one can see that $b(G)=\omega(G)/2$ by removing a perfect matching of $G$. When $G$ is an odd clique, then one can see that $b(G)=(\omega(G)-1)/2 + 1$ by removing a perfect matching of $G$ and any edge incident to the remaining universal vertex. So when $G$ is a clique, then $b(G) = \lceil \omega(G)/2 \rceil$. Therefore we can assume that $G$ is not a clique and so $\omega(G)\leq \Delta(G)$. \\

   For the sake of contradiction, we suppose that $b(G)>\omega(G)$. Let $K$ be a clique of $G$. The \textit{partition distance} in $G$ with respect to $K$ is the partition $(A_0,\ldots,A_k)$ of $V$ such that $A_0=V(K)$ and $A_i=\{v\in V \mid v\in N(u), u\in A_{i-1}\}$, for $i=1,\ldots,k$. Note that $A_i$ is the set of vertices at distance $i$ from $K$.

   \begin{claim}\label{c1}
      Let $C\subseteq A_i$ where $i\neq 0$, be such that $G[C]$ is a connected component of $G[A_i]$, and let $Q=N(C)\cap A_{i-1}$. Then $G[Q]$ is a clique.
    \end{claim}

   For contradiction, suppose that $G[Q]$ is not a clique. Since $A_0$ is a clique, we can consider that $i\geq 2$. Let $u,u'\in Q$ such that $uu'\not\in E$. There is a path from $u$ to $K$ and from $u'$ to $K$ in $G[A_0\cup\ldots\cup A_{i-2}\cup \{u,u'\})]$. Therefore there is an induced path $P=u-\cdots-u'$ from $u$ to $u'$ in $G[A_0\cup\ldots\cup A_{i-2}\cup \{u,u'\})]$. Let $P'=u-\cdots-u'$ be an induced path from $u$ to $u'$ in $G[C\cup\{u,u'\}]$. Then $G[V(P)\cup V(P')]$ is an induced cycle of length at least four, a contradiction. So $G[Q]$ is a clique. This proves Claim \ref{c1}. \\

   Let $W\subseteq A_i$, where $i=0,\ldots,k$, such that $G[W]$ is a connected component of $G[A_i]$ with at least two vertices. We restrict $W$ such that $F=N(W)\cap A_{i+1}$ is either empty or an independent set of $G$, and such that $N(F)\cap A_{i+2}=\emptyset$. We choose $W$ so that $\psi(K)=\vert F\cup W\vert$ is minimum. When $W\neq V(K)$, we denote $Q=N(W)\cap A^{i-1}(K)$. Note that when $W=V(K)$, then $Q=\emptyset$.

   We show that $W$ exists such as described above. Since $G$ is not a clique, it follows that $A_{k-1},A_{k}\neq \emptyset$. If $A_{k}$ is not an independent set of $G$, then there is a connected component $C$ of $G[A_k]$ with at least two vertices. Since $\vert C\vert \geq 2$ and $N(C)\cap A_{k+1}=\emptyset$, it follows that $W$ exists. Now we can assume that $A_k$ is an independent set of $G$. Let $C$ be a connected component of $G[A_{k-1}]$ such that $N(C)\cap A_k\neq \emptyset$. If $\vert C\vert \geq 2$, then $W$ exists since $N(C)\cap A_k$ is an independent set of $G$ and $A_{k+1}=\emptyset$. Hence it remains the case where $\vert C\vert = 1$. Let $C=\{u\}$ and $v\in N(u)\cap A_k$. From Claim \ref{c1} $G[N(v)\cap A_{k-1}]$ is a clique. Thus $N(v)=\{u\}$ and $d(v)=1$. From Claim \ref{c1} $N(u)\cap A_{k-2}$ is a clique. Therefore $d(u)\leq \omega(G)$. Then from Theorem \ref{ubound} $b(G)\leq d(u) + d(v) - 1\leq \omega(G)$, a contradiction. Hence $\vert C\vert \geq 2$ and so $W$ exists. \\

   Let $K$ be a clique of $G$ such that $\psi(K)=min(\{\psi(K') \mid K'$ \textit{is a clique of} $G\})$. We consider the sets $A_0,\ldots, A_k,F,Q,W$ as described above in the partition distance with respect to $K$.

   \begin{claim}\label{c2}
      For every $u\in W$ such that $Q = N(u)\cap A_{i-1}$, the sets $W\setminus \{u\}$ and $N(u)\cap (F\cup W)$ are independent in $G$, and $W = N[u]\cap W$.
   \end{claim}

   For contradiction, suppose that $W\setminus \{u\}$ or $N(u)\cap (F\cup W)$ is not an independent set of $G$. Let $K'=G[Q\cup \{u\}]$. Note that $Q$ is empty when $W=A_0$. From Claim \ref{c1} $Q$ is a clique and it follows that $K'$ is also a clique. Let $A_0',A_1',\ldots,A_{k'}'$ be the partition distance with respect to $K'$. Hence $A_0'=K'$. Since $W\setminus \{u\}$ or $N(u)\cap (F\cup W)$ are not an independent set, there is $W'\subseteq A_1'\cap (F\cup W)$ such that $W'$ is a connected component of $G[A_1']$ with at least two vertices. Let $F'=N(W')\cap A_2'$. Note that $F'\subseteq F$. Therefore either $F'=\emptyset$ or $F'$ is an independent set of $G$, and $N(F')\cap A_3'=\emptyset$. Then $\vert F'\cup W'\vert \leq \vert F\cup W\vert -1$ and thus $\psi(K)$ is not minimum, a contradiction. Hence $W\setminus \{u\}$ and $N(u)\cap (F\cup W)$ are two independent sets of $G$. Since $G[W]$ is connected, it follows that $W\subseteq N[u]$. This proves Claim \ref{c2}.

   \begin{claim}\label{c3}
      There exists $u\in W$ such that $Q = N(u)\cap Q$.
   \end{claim}

   For contradiction, suppose that for every vertex $u\in W$, we have $Q\neq N(u)\cap Q$ i.e.\ $Q\not\subseteq N(u)$. Let $u\in W$ such that $\vert N(u)\cap Q\vert$ is maximal. Since every vertex of $Q$ has a neighbor in $W$, there is $u'\in W$ such that $q'u'\in E$ and $q'u\not\in E$, where $q'\in Q$. We choose $u'$ so that $d(u,u')$ is minimal. From the maximality of $\vert N(u)\cap Q\vert$, there is $q\in Q$ such that $qu\in E$ and $qu'\not\in E$. Since $G[W]$ is connected, there is a shortest path $P=u-\cdots-u'$ between $u$ and $u'$ in $G[W]$. If $P=u-u'$, then $C_4=q-q'-u'-u-q$ is an induced cycle of length four, a contradiction. Let $v\in V(P)\setminus \{u,u'\}$. Suppose that $q'v\in E$. From the minimality of $d(u,u')$, it follows that $N(u)\cap Q\subseteq N(v)\cap Q$. Then $\vert N(v)\cap Q\vert > \vert N(u)\cap Q\vert$ is a contradiction of the maximality of $\vert N(u)\cap Q\vert$. Hence for every $v\in V(P)\setminus \{u,u'\}$, we have $vq'\not\in E$. Therefore if no vertex of $V(P)\setminus \{u,u'\}$ is a neighbor of q, it follows that $G[V(P)\cup \{q,q'\}]$ is an induced cycle of length at least five, a contradiction. So there is $v\in V(P)\setminus \{u,u'\}$ such that $qv\in E$. We choose $v$ such that $d(u',v)$ is minimum. Let $P'=v-\cdots-u'$ be a shortest path between $u'$ and $v$. Then $G[V(P')\cup \{q,q'\}]$ is an induced cycle of length at least four, a contradiction. This proves Claim \ref{c3}.

   \begin{claim}\label{c4}
      For every $u\in W$, $\vert N(u)\cap F\vert \leq 1$, and for every $v\in F$, $d(v)=1$.
   \end{claim}

   For contradiction, suppose there exists $u\in W$ such that $v,v'\in N(u)\cap F$. From Claim \ref{c3} there is $w\in W$ such that $Q=N(w)\cap Q$. From Claim \ref{c2} $W=N[w]\cap W$, and $W\setminus \{w\}$, $(F\cup W)\cap N(w)$ are two independent sets of $G$. From Claim \ref{c1} $N(v)\cap A_i$, $N(v')\cap A_i$ are two cliques and therefore $N(v)\subseteq W$ and $N(v')\subseteq W$. If $d(v)\geq 2$ or $d(v')\geq 2$, then $(F\cup W)\cap N(w)$ is not an independent set. Hence $d(v), d(v')\leq 1$. Yet from Theorem \ref{ubound} it follows that $b(G)\leq d(v) + d(v') - 1\leq 1$, a contradiction. This proves Claim \ref{c4}.

   \begin{claim}\label{c5}
      $\vert Q\vert \leq \omega(G) - 1$
   \end{claim}

   From Claim \ref{c1} $Q$ is a clique and from Claim \ref{c3} there is $u\in W$ such that $Q=N(u)\cap Q$. Hence $Q\cup \{u\}$ is a clique and therefore $\vert Q\vert \leq \omega(G) - 1$. This proves Claim \ref{c5}. \\

   From Claim \ref{c3} there is $u\in W$ such that $Q = N(u)\cap Q$. Recall that $\vert W\vert \geq 2$ and that $G[W]$ is a connected. Suppose that there is $v\in W$, $u\neq v$, such that $Q = N(v)\cap Q$. From Claim \ref{c2} $W\setminus \{u\}$ and $W\setminus \{v\}$ are two independent sets of $G$. Thus $W=\{u,v\}$. From Claim \ref{c1} $Q$ is a clique, and therefore $\vert Q\vert \leq \omega(G)-2$. From Claim \ref{c4} $\vert N(u)\cap F\vert, \vert N(v)\cap F\vert\leq 1$. Hence $d(u) \leq \vert Q\cup W\setminus \{u\}\vert + 1\leq \omega(G)$ and $d(v)\leq \vert Q\cup W\setminus \{v\}\vert + 1\leq \omega(G)$. Suppose that $u$ has a neighbor $x\in F$. It follows from Claim \ref{c4} that $d(x)=1$. Thus from Theorem \ref{ubound} $b(G)\leq d(u) + d(x) - 1\leq \omega(G)$, a contradiction. Hence $N(u)\cap F, N(v)\cap F=\emptyset$. Therefore $d(u)=d(v)= \omega(G)-1$. From Theorem \ref{uboundneigh} it follows that $b(G)\leq d(u) + d(v) - 1 - \vert N(u)\cap N(v)\vert \leq \omega(G)$, a contradiction. \\

   So we can assume that $u$ is the only vertex in $W$ such that $Q = N(u)\cap Q$. We show that $F$ is empty. Recall that from Claim \ref{c1} $G[Q]$ is a clique, from Claim \ref{c5} $\vert Q\vert \leq \omega(G) - 1$, and from Claim \ref{c4} every vertex of $W$ has at most one neighbor in $F$. Moreover from Claim \ref{c2} $W=N[u]$ and $(F\cup W)\setminus \{u\}$ is an independent set of $G$. Hence for every $v\in W\setminus \{u\}$, we have $d(v)\leq \vert Q\vert + 1\leq \omega(G)$. Let $x\in F$. From Claim \ref{c4} $d(x)=1$. If there is $v\in W\setminus \{u\}$ a neighbor of $x$, then from Theorem \ref{ubound} it follows that $b(G)\leq d(v) + d(x) - 1\leq \omega(G)$, a contradiction. Hence $x$ is a neighbor of $u$. Yet for every $v\in W\setminus \{u\}$, we have $d(v,x)\leq 2$. Therefore from Theorem \ref{ubound} it follows that $b(G)\leq d(v) + d(x) - 1\leq \omega(G)$, a contradiction. Hence $F=\emptyset$. It follows that for every $v\in W\setminus \{u\}$, we have $d(v)\leq \vert Q\vert \leq \omega(G)-1$. \\

   Let $S$ be a minimum dominating set of $G$. Suppose that $\vert S\cap W\vert \geq 2$. Then $(S\setminus W)\cup \{u\}$ is a dominating set, a contradiction. Hence for every minimum dominating set of $G$, we have $\vert S\cap W\vert \leq 1$. Let $v\in W\setminus \{u\}$ and $E_v=\{vv'\in E \mid v'\in N(v)\}$. Recall that $d(v)\leq \omega(G) - 1$, and therefore $\vert E_v\vert \leq \omega(G) - 1$. Let $w\in W\setminus \{v\}$ ($u=w$ is possible). Let $E_w=\{qw\in E \mid q\in (N(w)\cap Q)\setminus N(v)\}$, that is, the edges incident to $w$ with an extremity in $Q$ that is not a neighbor of $v$. Note that $\vert E_w\vert \leq \vert Q\setminus N(v)\vert$, and therefore $\vert E_v\cup E_w \vert \leq \vert Q\vert + 1\leq \omega(G)$. We remove the edges $E_v\cup E_w$ from $G$ to construct $G'=(V,E - (E_v\cup E_w))$. Since $b(G)>\omega(G)$, it follows that $\gamma(G') = \gamma(G)$. Let $S'$ be a minimum dominating set of $G'$. Since $G'$ is the graph $G$ minus some edges, any dominating set of $G'$ is a dominating set of $G$. Hence $S'$ is a minimum dominating set of $G$. Therefore from previous arguments, we have $\vert S'\cap W\vert \leq 1$. Note that $v$ is isolated in $G'$, and thus $v\in S'$. If $S'\cap N_G(v)\neq \emptyset$, then $S'\setminus \{v\}$ is a dominating set of $G$, a contradiction. Hence $S'\cap N_G(v)=\emptyset$. Recall that $N_{G'}(w)\cap Q\subseteq N_G(v)\cap Q$. Hence $N_{G'}(w)\cap S'\cap W\neq \emptyset$. Yet it follows that $\vert S'\cap W\vert \geq 2$, a contradiction.

   Hence $\gamma(G')>\gamma(G)$. Since we removed at most $\omega(G)$ edges from $G$ to construct $G'$, it follows that $b(G)\leq \omega(G)$. This completes the proof.
\end{proof}

We show that the bound of Theorem \ref{BondChordal} is sharp. The corona $G_1 \circ G_2$ (introduced by Frucht and Harary in \cite{Frucht}) is the graph formed from $\vert V(G_1)\vert$ copies of $G_2$ by joining the ith vertex of $G_1$ to the ith copy of $G_2$. Let $G=K_n \circ K_1$. Note that $\omega(G)=\Delta(G)=n$. Carlson and Develin in \cite{Carlson} have shown that  $\gamma(G) =\omega(G)$ and that $b(G)=\omega(G)$.

For non-chordal graphs, we show that there is an infinite family of graphs $\mathcal{C}$, where for every $G\in \mathcal{C}$, we have $b(G)>\omega(G)$, and its longest induced cycle has length four. The \textit{cartesian product} $G\,\square \, H$ of two graphs $G$ and $H$ is the graph whose vertex set is $V(G)\times V(H)$. Two vertices $(g_1,h_1)$ and $(g_2,h_2)$ are adjacent in $G\,\square\, H$ if either $g_1=g_2$ and $h_1h_2$ is an edge in $H$ or $h_1=h_2$ and $g_1g_2$ is an edge in $G$. Consider $G=(P_2\, \square\, P_k) \circ K_1$, where $k\geq 2$. The longest cycle of $G$ is four and $\omega(G)=2$. Then one can easily check that $\gamma(G)=2k$ and that $b(G)=3=\omega(G)+1$. We remark that it would be of interest to know if there exists a graph $G$ for which the longest cycle is $C_4$, and such that $b(G)>\omega(G)+1$. Graphs for which the longest cycle is $C_4$ may be known as the class of \textit{quadrangulated} graphs (an extension of chordal graphs, that is, chordal graphs where $C_4$ are allowed). \\

Since for a planar graph $G$, we have $\omega(G)\leq 4$, we obtain the following bound:

\begin{coro}\label{BondPlanar}
   Let $G$ be a planar chordal graph. When $G$ is not a clique, then $b(G)\leq 4$. If $G=K_2$, then $b(G)=1$. If $G=K_3$ or $G=K_4$, then $b(G)=2$.
\end{coro}

We remark that Corollary \ref{BondPlanar} may be of used to tackle the following conjecture of Dunbar et al.\ on the bondage number of planar graphs (see Chapter 17 p. 475, Conjecture 17.10 of \cite{Dunbar}).

\begin{conj}{\cite{Dunbar}}
   If $G$ is a planar graph, then $b(G)\leq \Delta(G) + 1$.
\end{conj}

We leave the following problem: \\

\noindent
\textbf{Problem : } Characterize the chordal graphs for which $b(G)=\omega(G)$.

\begin{ack}
   The author would like to thank Christophe Picouleau, St\'ephane Rovedakis and Fran\c cois  Delbot for providing helpful comments.
\end{ack}

\end{document}